\documentclass{amsart}
\usepackage{amssymb}
\usepackage{graphicx}
\usepackage{framed}
\usepackage{color}
\newtheorem{thm}{Theorem}[section]
\newtheorem{cor}[thm]{Corollary}
\newtheorem{lem}[thm]{Lemma}
\newtheorem{prop}[thm]{Proposition}

\theoremstyle{definition}
\newtheorem{defn}[thm]{Definition}
\theoremstyle{remark}
\newtheorem{rem}[thm]{Remark}

\newtheorem{ex}[thm]{Example}
\numberwithin{equation}{section}
\newcommand{\norm}[1]{\left\Vert#1\right\Vert}

\newcommand{\To}{\longrightarrow}

\usepackage{marginnote}

\begin{document}

\title[Nonexistence of QC maps between certain spaces]{Nonexistence of quasiconformal maps between certain metric measure spaces}

\renewcommand{\subjclassname}{%
 \textup{2010} Mathematics Subject Classification}

\author[F\"assler]{Katrin F\"assler}

\address[F\"assler]{Department of Mathematics and Statistics, Gustaf H\"allstr\"omin
 katu 2b, FI-00014-University of Helsinki, Finland}
\email{katrin.fassler@helsinki.fi}

\author[Koskela]{Pekka Koskela}

\author[Le Donne]{Enrico Le Donne}

\address[Koskela and Le Donne]{Department of Mathematics and Statistics, P.O. Box 35,
FI-40014,
University of Jyv\"askyl\"a, Finland}%

\email{pkoskela@maths.jyu.fi}
\email{enrico.ledonne@jyu.fi}

\thanks{K.F. was supported by the Academy of Finland, project number 252293.}

\maketitle

\begin{abstract}
We provide new conditions that  ensure that two metric measure spaces are not quasiconformally equivalent. As an application we deduce that there exists no quasiconformal map between the sub-Riemannian Heisenberg and roto-translation groups.
\end{abstract}

\section{Introduction}

The metric definition of quasiconformality can be formulated for maps between  arbitrary metric spaces, and a rich theory has been developed on metric measure spaces with controlled geometry, see for example
\cite{MR1654771,MR1869604,Balogh_Koskela_Rogovin,Williams, HKNT_book}. In this context, one would like to decide whether two given spaces $(X,d_X,\mu_X)$ and $(Y,d_Y,\mu_Y)$ are quasiconformally equivalent. In the case where $X$ and $Y$ are Carnot groups (endowed with their sub-Riemannian distances),
Pansu has shown in \cite{MR979599} that they are quasiconformally homeomorphic if and only if they are isomorphic. The reason comes from the fact that a quasiconformal map between two Carnot groups is differentiable almost everywhere and the differential is a group isomorphism. Margulis and Mostow \cite{MR1334873} have generalized Pansu's differentiability theorem to a vast class of sub-Riemannian manifolds.
As a consequence, if two sub-Riemannian Lie groups are quasiconformally equivalent, then necessarily the respective tangent cones have to be isomorphic.

Yet this is not a sufficient condition. As an example we will present two sub-Riemannian Lie groups which have the same tangent cones yet are not (globally) quasiconformally equivalent. The first one is the standard sub-Riemannian Heisenberg group $\mathbb{H}^1$. The second one is the universal cover of  $SE(2)$, i.e., the group of orientation-preserving isometries of the Euclidean $2$-space. When endowed with the standard left-invariant sub-Riemannian structure we call such a metric space the (universal cover of the sub-Riemannian) \emph{roto-translation group} and denote it by $\mathcal{RT}$. The space $\mathcal{RT}$ is not a Carnot group and its tangent cone at every point is $\mathbb{H}^1$.

Another general obstruction to  the existence of a quasiconformal homeomorphism is a
different capacity at infinity. Namely, if two
 metric spaces of locally $Q$-bounded geometry are quasiconformally equivalent and one of them has zero $Q$-capacity at infinity then the other one has
 zero $Q$-capacity at infinity  as well, see \cite{MR1654771, Zorich_GAFA, HKNT_book}.
Such a fact can be used to prove that
the Riemannian $m$-th Heisenberg group $\mathbb{H}^m$
and
the Euclidean space
$\mathbb{R}^{2m+1}$
are not quasiconformally equivalent, since
$\mathbb{H}^m$ has positive $(2m+1)$-capacity at infinity while the same capacity for $\mathbb{R}^{2m+1}$ is zero.
 Actually,
it can be shown that the only quasiregular maps  from $\mathbb{R}^{2m+1}$
to the Riemannian $\mathbb{H}^m$ are constant; see the discussion in \cite[p.627]{aHoRi92}.
Regarding the problem of how to show that the two sub-Riemannian spaces
$\mathcal{RT}$ and $\mathbb{H}^1$
are not
quasiconformally equivalent,
the method of looking at the capacity at infinity fails. Indeed,
both spaces are  of locally $4$-bounded geometry and their  $4$-capacity at infinity is zero
since the volume of balls grows at most as the $4$-th power of the radius, see
\cite{MR1654771, HKNT_book}.
In addition, we point out that
   these spaces exhibit different volume growths at large scale. The different geometric behaviour at small and large scales a priori does not rule out the existence of quasiconformal maps.
   Indeed, it is easy to give examples of
   quasiconformally equivalent spaces with different volume growth on the large. See Section
 \ref{examples} where   examples are discussed.


So in general, the existence of a quasiconformal map $f:X\to Y$ between spaces of locally $Q$-bounded geometry is possible, even if
$Y$  has volume growth with exponent $Q$ at large scale and
$X$ has volume growth with exponent $N$ at large scale with $N < Q$.
Imposing the additional condition that $X$ contains a continuously and quasi-isometrically embedded copy of $\mathbb{R}$ and that $Y$ is proper and has a Loewner function blowing up at  zero, we shall prove that there cannot exist a quasiconformal homeomorphism between $X$ and $Y$. This is the content of Theorem \ref{t:main_qi_embed} below, which applies in particular to $X=\mathcal{RT}$ and $Y=\mathbb{H}^1$ endowed with their standard sub-Riemannian metrics.

\begin{thm}\label{t:main_qi_embed}
Let $Q>1$. Assume that $X=(X,d_X,\mu_X)$ is a metric measure space such that
\begin{itemize}
\item $X$ is of locally $Q$-bounded geometry,
\item there exists a continuous quasi-geodesic $\sigma:\mathbb{R}\to X$,
\item there exists
$R_0>0$, $N< Q$, and $C_0 >0$ such that
\begin{equation}\label{lower:volume:growth}
\mu(B(\sigma(0),r))\leq C_0 r^N,\quad \text{for all }r\geq R_0.
\end{equation}
\end{itemize}
Let further $Y=(Y,d_Y,\mu_Y)$ be a metric measure space with Loewner function $\phi_Q$ such that
\begin{itemize}
\item $Y$ is proper, i.e., its closed balls are compact,
\item $Y$ is of locally $Q$-bounded geometry,
\item $\lim_{t\to 0} \phi_Q(t)=\infty$.
\end{itemize}
Then $X$ and $Y$ are not quasiconformally equivalent.
\end{thm}

We will show how to deduce from the theorem the following consequence.
\begin{cor}\label{RTvsH}
The sub-Rie{\-}mannian roto-translation group $\mathcal{RT}$ and the sub-Rie{\-}mannian Heisenberg group $\mathbb{H}^1$ are not quasiconformally equivalent.
\end{cor}

The paper is organized as follows. In the next section we recall the  definitions
mentioned in Theorem \ref{t:main_qi_embed}
 and some useful general results.
In Section \ref{sec:proof} we prove the main result: Theorem \ref{t:main_qi_embed}, and we deduce few consequences. 
Section \ref{examples} is devoted to show  how sharp Theorem \ref{t:main_qi_embed} is. Namely,  we illustrate that it is not possible to remove from  the assumptions  the existence of a quasi-geodesic, or the properness of the range, or the divergence of the Loewner function. Finally, in Section \ref{Application} we recall the definitions and some properties of the sub-Riemannian Heisenberg group
and the sub-Riemannian roto-translation group. We end with the proof of Corollary \ref{RTvsH}.

\subsection*{Acknowledgements}
Corollary \ref{RTvsH} answers a question posed to us by
 Kirsi Peltonen.
 The discussion with her was the initial motivation for the research leading to this paper.
 We thank her for sharing with us the problem in question.
\section{Preliminaries}

\begin{defn}\label{def:quasiconformal}
A  \emph{quasiconformal} map between two metric spaces $(X,d_X)$ and $(Y,d_Y)$ is a homeomorphism $f:X\to Y$ for which there exists a finite constant $K\geq 1$ such that for all $x\in X$,
\begin{displaymath}
H(x):=\limsup_{r \to 0}\frac{\sup_{d_X(x,x')\leq r}d_Y(f(x),f(x'))}{\inf_{d_X(x,x')\geq r}d_Y(f(x),f(x'))}\leq K.
\end{displaymath}
\end{defn}
An important tool in the study of quasiconformal maps is the modulus of a curve family.
\begin{defn}\label{def:modulus} Let $\Gamma$ be a family of curves in a metric measure space $(X,d,\mu)$, where $\mu$ is a nontrivial Borel regular measure.
A Borel function $\rho:X \to [0,\infty]$ is said to be \emph{admissible} for $\Gamma$, and we write $\rho \in \mathrm{adm}(\Gamma)$, if $\int_{\gamma} \rho ds \geq 1$ for all locally rectifiable $\gamma \in \Gamma$. The \emph{$Q$-modulus}, for $1\leq Q<\infty$, of $\Gamma$ is then defined as
\begin{displaymath}
M_Q(\Gamma):=\inf_{\rho\in\mathrm{adm}(\Gamma)} \int \rho^Q \;\mathrm{d}\mu.
\end{displaymath}
\end{defn}
In \cite{MR1869604} it is shown that the $Q$-modulus is a quasi-invariant for quasiconformal maps between spaces of locally $Q$-bounded geometry, see Theorem \ref{quasi-invariant:modulus} below. We recall Definition 9.1 from \cite{MR1869604} with the modification as in Remark 9.4(b).
\begin{defn}\label{bounded geometry}
A metric measure space $(X,d,\mu)$ is of \emph{locally $Q$-bounded geometry}, $Q>1$, if $X$ is separable, pathwise connected, locally compact, and if there exists a constant $C_0 \geq 1$ and a decreasing function $\phi: (0,\infty) \to (0,\infty)$ such that each point in $X$ has a neighbourhood $U$ (with compact closure in $X$) so that
\begin{enumerate}
\item  $\mu(B_R)\leq C_0 R^Q$ whenever $B_R \subset U$ is a ball of radius $R>0$
\item $M_Q(\Gamma_{E,F})\geq \phi(t)$ whenever $B_R \subset U$ is a ball of radius $R>0$ and $E$, $F$ are two 
 continua in $B_R$ with $$0<\mathrm{dist}(E,F) \leq t \cdot \mathrm{min}\{\mathrm{diam}E,\mathrm{diam}F\}.$$ Here, $\Gamma_ {E,F}$ denotes the family of closed paths joining $E$ and $F$, that is, it consists of all continuous functions $\gamma:[0,1]\to X$ such that  $\gamma(0)\in E$ and $\gamma(1)\in F$.
\end{enumerate}
\end{defn}


\begin{thm}[{\cite[Theorem 9.8]{MR1869604}}]\label{quasi-invariant:modulus}
If $f:X\to Y$ is a homeomorphism between two spaces of locally $Q$-bounded geometry, with $Q>1$, then $f$ is quasiconformal (as in Definition \ref{def:quasiconformal}) if and only if there exists a constant $K'>0$ for which
$$\dfrac{1}{K'}M_Q(\Gamma)<M_Q(f (\Gamma)) <K' M_Q(\Gamma) ,$$
for all curve families $\Gamma$ in $X$.
\end{thm}

To ensure that the function  $\phi(t)$ in Definition \ref{bounded geometry} goes to $\infty$ as $t\to 0$, one can assume that the metric measure space is
Ahlfors $Q$-regular
and
$Q$-Loewner, see Theorem \ref{t:growth}.

\begin{defn} Let $Q>1$.
A metric measure space $(X,d,\mu)$ is \emph{Ahlfors $Q$-regular} if $\mu$ is a Borel regular measure on $X$ such that there exists a constant $C>0$ such that
for all closed balls $B_R$ of radius $0<R<\mathrm{diam}X$,
\begin{displaymath}
C^{-1} R^Q \leq \mu(B_R) \leq C R^Q.
\end{displaymath}
\end{defn}

\begin{defn}
Let $Q>1$ and let $(X,d,\mu)$ be a metric measure space. The {\em Loewner function}  is defined as
\begin{displaymath}
\phi_Q(t):= \inf M_Q(\Gamma_{E,F}),\quad\text{for all }t>0,
\end{displaymath}
where the infimum is taken over all 
 continua $E,F\subseteq X$ with
\begin{displaymath}
0<\mathrm{dist}(E,F) \leq t \min \{\mathrm{diam}E,\mathrm{diam}F\}.
\end{displaymath}
We call $(X,d,\mu)$  a \emph{$Q$-Loewner space} if it is pathwise connected and the  Loewner function is strictly positive.

\end{defn}

\begin{thm}[{\cite[Theorem 3.6]{MR1654771}}]\label{t:growth}
Let $Q>1$ and let $(X,d,\mu)$ be an Ahlfors $Q$-regular $Q$-Loewner space. Then the Loewner function behaves asymptotically as
\begin{displaymath}
\phi_Q(t)\simeq \log(1/t),\quad t\to 0.
\end{displaymath}
\end{thm}

Regarding Theorem \ref{t:growth}, we also refer the reader to the comments in   \cite[Section 8]{bHe01}. 
 In our application, Theorem \ref{t:growth} will ensure  that for a particular sequence of curve families $(\Gamma_n')_n$ in the target space $Y$, the corresponding sequence of moduli $M_Q(\Gamma_n')$ tends to infinity as $n\to \infty$. At the same time we impose a condition on the source space $X$ guaranteeing that $M_Q(\Gamma_n)$ with $\Gamma_n'=f(\Gamma_n)$ is uniformly bounded for any homeomorphism $f:X\to Y$, whence $f$ cannot be quasiconformal. Such an extra condition on    $X$ is to have volume growth with exponent $N<Q$ at large scale and to contain a continuously and quasi-isometrically embedded copy of $\mathbb{R}$.

\begin{defn} Let $(X,d_X)$ and $(Y,d_Y)$ be metric spaces.
A not necessarily continuous map $h:X\to Y$ is a \emph{quasi-isometric embedding} if there exist constants $L>0$ and $b>0$ such that for all $x,x'\in X$, one has
\begin{displaymath}
L^{-1} d_X(x,x')-b\leq d_Y(h(x),h(x'))\leq L d_X(x,x')+b.
\end{displaymath}
\end{defn}

A quasi-isometric embedding $\sigma:\mathbb{R}\to Y$ is called a {\em quasi-geodesic} of $Y$.
We point out that, in the case when $Y$ is a length space, the presence of a
quasi-geodesic ensures the
existence of a {\em continuous}
quasi-geodesic, see Lemma \ref{continuous:gamma}.

%

\section{Proof of the main theorem and some consequences}\label{sec:proof}

 We start by explaining the idea of the proof of
Theorem \ref{t:main_qi_embed}.
We shall use Theorem \ref{quasi-invariant:modulus}, i.e,  the quasi-invariance of the $Q$-modulus under quasiconformal maps between spaces of locally $Q$-bounded geometry.
 We  consider an arbitrary homeomorphism $f:X\to Y$. We will provide a nested sequence of curve families $\Gamma_1 \subseteq \Gamma_2 \subseteq \Gamma_3 \subseteq \ldots\subseteq \Gamma$ in $X$ such that
\begin{equation}\label{eq:mod_goal}
M_Q(\Gamma_n)\leq M_Q(\Gamma)<\infty\quad\text{and}\quad \lim_{n\to \infty}M_Q(f(\Gamma_n))=\infty.
\end{equation}
More precisely, we give sequences of continua $(E_n)_n$ and $(F_n)_n$ in $X$ so that for each $n\in \mathbb{N}$, the set $E_n$ is disjoint from  $F_n$ and $E_n\subseteq E_{n+1}\subseteq E$, $F_n\subseteq F_{n+1}\subseteq F$, where $E=\cup_n E_n$ and $F=\cup_n F_n$ are unbounded sets in $X$ such that $\Gamma$ is the family  of all closed paths connecting $E$ and $F$ and has finite $Q$-modulus. This will imply \eqref{eq:mod_goal}, which shows that $f$ cannot be quasiconformal according to Theorem \ref{quasi-invariant:modulus}.

The idea is to choose $E$ and $F$ as disjoint rays on an unbounded quasi-geodesic curve $\sigma$ -- which exists by assumption. Then, roughly speaking, because $\sigma$ is a quasi-geodesic, any curve $\gamma$ joining $E$ to $F$ cannot be too short. This will give that a certain density $\rho$ in $L^Q(X,\mu_X)$ is admissible for $\Gamma$ and thus
\begin{displaymath}
M_Q(\Gamma)\leq \int_X \rho^Q\;\mathrm{d}\mu_X <\infty.
\end{displaymath}
At the same time, since $f$ is assumed to be a homeomorphism, the sets $f(E_n)$ and $f(F_n)$ are disjoint nondegenerate continua in $Y$. As $Y$ is proper and the sets $E$ and $F$ are unbounded, we get $\lim_{n\to \infty} \mathrm{diam} (f(E_n))=\infty$ and $\lim_{n\to \infty} \mathrm{diam}( f(F_n))=\infty$. By the asymptotic behavior of the Loewner function $\lim_{n\to \infty}M_Q(f(\Gamma_n))=\infty$ for the family $\Gamma_n$ of closed paths connecting $E_n$ to $F_n$.

We now make these steps more precise.

Since $\sigma:\mathbb{R}\to X$ is quasi-isometric, there exist constants $b>0$ and $L\geq 1$ such that
\begin{displaymath}
L^{-1}|t-t'|-b \leq d_X(\sigma(t),\sigma(t'))\leq L|t-t'|+b
\end{displaymath}
for all $t,t'\in \mathbb{R}$. Notice that even if the inequalities  hold with $b=0$, that is, if we had a bi-Lipschitz embedding, we still choose a positive constant $b$ for later use.
Hence, if $t\in (-\infty,  -Lb]$ and $t'\in [ Lb,\infty)$, then $|t-t'|\geq 2Lb$ and thus $d_X(\sigma(t),\sigma(t'))\geq b>0$. In other words,
\begin{equation}\label{eq:cont_dist}
\mathrm{dist}(\sigma((-\infty,-Lb]),\sigma([Lb,+\infty)))\geq b.
\end{equation}
Next, we set
\begin{equation}\label{eq:d_R_0}
R_1:= \max\{R_0,2b(L^2+2)\}\quad \text{and } t_1:=L(b+R_1).
\end{equation}
The condition $R_1 \geq R_0$ ensures that $\mu_X(B(x_0,r))\leq C_0 r^N$ for $r\geq R_1$, where $x_0:= \sigma(0)$. The choice of the second term in the maximum will become clear later; eventually it guarantees that the length of a curve $\gamma \in \Gamma$ is appropriately bounded from below.

For $t\in (-\infty, -t_1]\cup [t_1,+\infty)$ we have that
\begin{displaymath}
d_X(\sigma(t),x_0)=d_X(\sigma(t),\sigma(0))\geq L^{-1}|t|-b\geq R_1,
\end{displaymath}
and thus
\begin{equation}
\sigma((-\infty, -t_1]\cup [t_1,+\infty))\subseteq X\setminus 
{B}(x_0,R_1).
\end{equation}
This motivates the   definitions
\begin{equation}\label{eq:E_n_F_n}
E_n:= \sigma([-n, -t_1]),\quad F_n:=\sigma([t_1,+n])
\end{equation}
and
\begin{equation}\label{eq:E_F}
E:= \sigma((-\infty, -t_1]),\quad F:= \sigma([t_1,+\infty)).
\end{equation}
We let $\Gamma$ be the family of all closed paths in $X$ connecting $E$ to $F$, and accordingly, $\Gamma_n \subseteq \Gamma$ shall consist of all closed paths connecting $E_n$ to $F_n$. Notice that \eqref{eq:cont_dist} ensures that the continua in the considered pairs are well-separated.

We plan to show that the $Q$-modulus of the just-defined  $\Gamma$ is finite.
For doing so we need to construct an admissible density and show that it  is $Q$-integrable.

 \begin{lem}\label{l:rho_adm_l:rho_int}
 There exists a Borel function $\rho$ such that
 \begin{displaymath}
 \rho \in \mathrm{adm}(\Gamma) \qquad \text{ and }\qquad
\int_X \rho^Q \; \mathrm{d}\mu_X < \infty.
\end{displaymath}
\end{lem}

  \begin{proof}

We set
 \begin{displaymath}
 c_0:= \frac{4}{b}\quad \text{and} \quad c_1:= 4(L^2 +1)
 \end{displaymath}
 and consider the density
 \begin{equation}\label{eq:def_density}
 \rho(x):= \left\{ \begin{array}{ll}c_0 & \text{ for } x\in B(x_0,R_1)\\ \frac{c_1}{d_X(x,x_0)}&\text{ for }x\in X\setminus B(x_0,R_1).\end{array}\right.
 \end{equation}
 The choice of $c_0$ and $c_1$ ensures that $\rho \in \mathrm{adm}(\Gamma)$ as we are going to   show.

Since the modulus of a family of curves does not depend on the parametrizations of the curves,
  without loss of generality, we may assume that each curve in $\Gamma$ is a rectifiable curve and is parametrized according to arc-length by a continuous function $\gamma:[0,\ell] \to X$ such that $\gamma(0)\in E$ and $\gamma(\ell)\in F$.
 Let $\gamma \in \Gamma$. First record that
 \begin{displaymath}
 \ell = \mathrm{length}(\gamma) \geq d_X(\gamma(0),\gamma(\ell))\geq b=\tfrac{4}{c_0}
 \end{displaymath}
 since $\gamma(0)\in E$, $\gamma(\ell)\in F$ and \eqref{eq:cont_dist} holds. Thus $1/c_0 < \ell/2$. Now there are two cases to consider: either
 \begin{displaymath}
 |\{s\in [0,\ell]:\; \gamma(s)\in B(x_0,R_1)\}|\geq \tfrac{1}{c_0},
 \end{displaymath}
 in which case trivially,
 \begin{displaymath}
 \int_{\gamma} \rho\;\mathrm{d}s\geq \int_{\{s\in [0,\ell]:\; \gamma(s)\in B(x_0,R_1)\}}c_0\;\mathrm{d}s\geq \frac{c_0}{c_0}=1,
 \end{displaymath}
 or   necessarily
  \begin{displaymath}
 |\{s\in [0,\ell]:\; \gamma(s)\in X\setminus  B(x_0,R_1)\}|\geq \ell-\tfrac{1}{c_0}.
 \end{displaymath}
 Set $M:= \sup_{s\in [0,\ell]} d_X(\gamma(s),x_0)$.
 By the choice of $c_0$ we have $\ell-\frac{1}{c_0} > \frac{\ell}{2}$ and hence by the above considerations,
 \begin{displaymath}
 \int_{\gamma} \rho\; \mathrm{d}s\geq \left(\ell-\frac{1}{c_0}\right)\frac{c_1}{M} \geq \frac{\ell c_1}{2M}.
 \end{displaymath}
 The admissibility of $\rho$ will thus be proven once we have shown that
 \begin{equation}\label{eq:goal_ell}
 \ell \geq \frac{2M}{c_1}.
 \end{equation}
To this end, we will use the fact that $\sigma$ is a quasi-isometric embedding. Since $s\mapsto d_X(\gamma(s),x_0)$ is a continuous function on a compact set, there exists $\bar s \in [0,\ell]$ such that
 \begin{equation}\label{e:def_M}
 M= d_X(\gamma(\bar s),x_0).
 \end{equation} Moreover, since $\gamma(0)\in E$ and $\gamma(\ell)\in F$, there exist
$s_1\in (-\infty, -t_1]$ and $s_2 \in [t_1,+\infty)$ such that
\begin{displaymath}
\gamma(0)=\sigma(s_1)\quad \text{and}\quad \gamma(\ell)=\sigma(s_2).
\end{displaymath}
Then, by quasi-isometry,
\begin{displaymath}
|s_1|+|s_2|=|s_2-s_1|\leq L(d_X(\sigma(s_1),\sigma(s_2))+b),
\end{displaymath}
which implies, again by quasi-isometry, that
\begin{align*}
2 d_X(\gamma(\bar s),x_0)&
= d_X(\gamma(\bar s),x_0) +d_X(\gamma(\bar s),x_0)\\
&\leq  d_X(\gamma (\bar s),\gamma (0))+d_X(\gamma (0),x_0)+d_X(\gamma (\bar s),\gamma (\ell))+d_X(\gamma (\ell),x_0)\\
&= d_X(\gamma (0),\gamma (\bar s))+d_X(\gamma (\bar s),\gamma (\ell))+d_X(\sigma (s_1),\sigma(0))+d_X(\sigma(0),\sigma(s_2))\\
&\leq d_X(\gamma (0),\gamma (\bar s))+d_X(\gamma (\bar s),\gamma (\ell))+ L(|s_1|+|s_2|)+ 2b\\
&\leq d_X(\gamma (0),\gamma (\bar s))+d_X(\gamma (\bar s),\gamma (\ell)) + L^2 (d_X(\sigma(s_1),\sigma(s_2))+b)+2b\\
&\leq (1+L^2) \left(d_X(\gamma (0),\gamma (\bar s))+d_X(\gamma (\bar s),\gamma (\ell))\right)+ b(L^2+2).
\end{align*}
Hence,
\begin{displaymath}
d_X(\gamma (0),\gamma (\bar s))+d_X(\gamma (\bar s),\gamma (\ell)) \geq \frac{2}{L^2+1} d_X(\gamma(\bar s),x_0)- b \frac{L^2+2}{L^2+1}.
\end{displaymath}
By the definition of $M$ as in \eqref{e:def_M}, we conclude that
\begin{align*}
\ell \geq d_X(\gamma (0),\gamma (\bar s))+d_X(\gamma (\bar s),\gamma (\ell))\geq \frac{2M}{L^2 +1} \left(1- \frac{b}{2M}(L^2 +2)\right)\geq \frac{M}{L^2 +1}\geq \frac{2M}{c_1},
\end{align*}
since $M\geq R_1 > b(L^2 +2)$ and $c_1 > 2(L^2 +1)$. Thus we have established \eqref{eq:goal_ell} as desired. This shows that
\begin{displaymath}
\int_{\gamma} \rho\;\mathrm{d}s \geq 1,\quad\text{for all }\gamma \in \Gamma,
\end{displaymath}
and thus proves the admissibility of the density  $\rho$ for the curve family $\Gamma$.

Next, we will show how the growth bound for $\mu_X$ ensures that the admissible density defined in \eqref{eq:def_density} belongs to $L^Q(X,\mu_X)$.
%
%
We use a consequence of Fubini's theorem, see \cite[1.15]{MR1333890}, to write
\begin{align*}
\int_X &\rho^Q \; \mathrm{d}\mu_X  \\&= \int_{B(x_0,R_1)}c_0^Q \;\mathrm{d}\mu_X + \int_0^{(c_1/R_1)^{Q}} \mu_X(\{x\in X\setminus B(x_0,R_1):\;(\tfrac{c_1}{d_X(x,x_0)})^Q\geq \eta\})\;\mathrm{d}\eta.
\end{align*}
In the sum above, the first integral  is finite, 
since $\mu_X(B(x_0,R_1))<\infty$.
The second integral is estimated from above as
 $$
\int_0^{(c_1/R_1)^{Q}} \mu_X(B(x_0,c_1 \eta^{-\frac{1}{Q}}))\;\mathrm{d}\eta
\leq C_0 c_1^N	 \int_0^{(c_1/R_1)^{Q}} \eta^{-\frac{N}{Q}} \;\mathrm{d}\eta,
$$
since in the integrand $c_1 \eta^{-\frac{1}{Q}} \geq R_1 \geq R_0$.
Here $C_0$ and $R_0$ are
 the constants of the large scale volume growth assumption \eqref{lower:volume:growth}.
Since $N<Q$, the last integral is finite, which concludes the proof of Lemma \ref{l:rho_adm_l:rho_int}.
\end{proof}

 \begin{prop}\label{p:mod_source}
 Assume that $X=(X,d_X,\mu_X)$ is a metric measure space that contains a quasi-geodesic and has a volume growth lower than $Q$ as in
 \eqref{lower:volume:growth}.
 If $E$ and $F$ are the sets defined in \eqref{eq:E_F},
 $E_n$ and $F_n$ are the sets defined in \eqref{eq:E_n_F_n},
and  $\Gamma$ (resp. $\Gamma_n$) is the family of all closed paths in $X$ connecting $E$ to $F$ (resp. connecting $E_n$ to $F_n$), then 
 \begin{eqnarray*}
  M_Q(\Gamma_n)\leq M_Q(\Gamma)<\infty\quad\text{ and }\quad
\dfrac{\mathrm{dist}(E_n,F_n) }{ \min \{\mathrm{diam}E_n,\mathrm{diam}F_n\}} \stackrel{n\to \infty}{\To} 0. 
 \end{eqnarray*}
 \end{prop}

 \begin{proof}
The only     statement  that is not a direct consequence of the definitions is the finiteness of the modulus of $\Gamma$. However, it  immediately follows  from
 Lemma~\ref{l:rho_adm_l:rho_int}. 
\end{proof}

Next, we will study the images $f(\Gamma_n)$ under a homeomorphism $f:X \to Y$. We remark that in the following proposition some assumptions are not actually necessary, e.g., the locally $Q$-bounded geometry and the volume growth
 \eqref{lower:volume:growth}.

\begin{prop}\label{p:mod_image}
 Let $X$ and $Y$ be as in Theorem \ref{t:main_qi_embed}.
If  $\Gamma_n$ is the family of all closed paths in $X$ connecting $E_n$ to $F_n$, which are defined in \eqref{eq:E_n_F_n}, we have
\begin{displaymath}
\lim_{n\to \infty} M_Q(f(\Gamma_n))= \infty,
\end{displaymath}
for any homeomorphism $f:X \to Y$.
\end{prop}

\begin{proof}
Since $\sigma$ is continuous and $[-n, -t_1],[t_1,+n]$ are compact connected sets, also the sets $E_n$ and $F_n$ in $X$ are compact and connected. Moreover, the definition and \eqref{eq:cont_dist} ensure that $E_n$ and $F_n$ are disjoint, in fact at distance at least $b$ from each other.
Since $\sigma$ is a quasi-isometric embedding of $\mathbb{R}$, the sets $E$ and $F$ are unbounded. As $f$ is a homeomorphism and $Y$ is proper, we must then also have
\begin{displaymath}
\lim_{n\to \infty}\mathrm{diam}f(E_n)=\lim_{n\to \infty} \mathrm{diam} f(F_n)=\infty.
\end{displaymath}
Since $f(E_1)\subseteq f(E_2)\subseteq f(E_3)\subseteq \ldots$ (and analogously for the sequence $(f(F_n))_n$), we have
\begin{displaymath}
\mathrm{dist}(f(E_n),f(F_n)) \leq \mathrm{dist} (f(E_1),f(F_1))<\infty
\end{displaymath}
and we find
\begin{displaymath}
\lim_{n\to \infty}\frac{\mathrm{dist}(f(E_n),f(F_n))}{\min\{\mathrm{diam}(f(E_n)),\mathrm{diam}((F_n))\}}=0,
\end{displaymath}
in other words, for each $t>0$, there exists $n(t)\in \mathbb{N}$ such that
\begin{align*}
\phi_Q(t)&:= \inf\left\{M_Q(\Gamma_{E',F'}):\; \tfrac{\mathrm{dist}(E',F')}{\min\{\mathrm{diam}E',\mathrm{diam}F'\}}\leq t\right\}\leq M_Q(f(\Gamma_n))
\end{align*}
for all $n\geq n(t)$. By assumption, we know that
$
\phi_Q(t)\to \infty,$ as $t\to 0.
$
This yields $\lim_{n\to \infty}M_Q(f(\Gamma_n))=\infty$ and thus concludes the proof of Proposition \ref{p:mod_image}.
\end{proof}

\begin{proof}[Proof of Theorem \ref{t:main_qi_embed}]
The theorem  follows from Proposition \ref{p:mod_source} and Proposition \ref{p:mod_image}. Indeed, if $f$ was quasiconformal, according to Theorem \ref{quasi-invariant:modulus} there should exist a finite constant $K'\geq 1$ such that
\begin{displaymath}
M_Q(f(\Gamma_n))\leq K' M_Q(\Gamma_n)\leq K 'M_Q(\Gamma)<\infty, \quad \text{for all }n\in \mathbb{N},
\end{displaymath}
which is impossible.
The proof of Theorem \ref{t:main_qi_embed} is concluded.
\end{proof}

We point out that Proposition \ref{p:mod_source},  used in the first part of the proof of Theorem~\ref{t:main_qi_embed}, gives the following fact, which we will  use later in Example \ref{ex:loewner} to show how necessary is the assumption on the Loewner function in Theorem~\ref{t:main_qi_embed}.


\begin{prop}\label{p:bdd_loewner}
Assume that $X=(X,d_X,\mu_X)$ is a metric measure space
that contains a quasi-geodesic and has a volume growth lower than $Q$ as in
 \eqref{lower:volume:growth}.
Then the $Q$-Loewner function $\phi_Q$ of $X$ is bounded.
\end{prop}

In Theorem \ref{t:main_qi_embed} we can replace the assumption that the quasi-geodesic is continuous with the stronger assumption that the space $Y$ is a length space, or even a geodesic space.

\begin{lem}\label{continuous:gamma}
If $(X,d_X)$ is a length space and $\widetilde{\sigma}:\mathbb{R}\to X$ a quasi-isometric embedding, then there exists a continuous quasi-isometric embedding $\sigma:\mathbb{R}\to X$.
\end{lem}

We leave the straightforward  proof of the above  lemma as an  exercise.

\begin{cor}\label{c:main_thm}
Given $Q>1$, let $X=(X,d_X,\mu_X)$ be a length space of locally $Q$-bounded geometry for which there exists a quasi-isometric embedding $\sigma:\mathbb{R}\to X$ and constants $R_0>0$, $N<Q$ and $C_0 >0$ such that
\begin{displaymath}
\mu(B(x,r))\leq C_0 r^N, \qquad \forall  x\in X, \forall r\geq R_0.
\end{displaymath}
Then $X$ is not quasiconformally equivalent to any proper Ahlfors $Q$-regular $Q$-Loewner space.
\end{cor}

\begin{proof}
This is an immediate consequence of Theorem \ref{t:main_qi_embed} if we apply Theorem \ref{t:growth} and Lemma \ref{continuous:gamma}.
Indeed, from Theorem \ref{t:growth}  we get that  the Loewner function blows up at the origin. 
To get a continuous quasi-geodesic we use Lemma \ref{continuous:gamma}, since $X$ is a length space.
\end{proof}

\section{Examples for the sharpness of the assumptions}\label{examples}
In this section we provide several examples to illustrate the sharpness of the assumptions in Theorem \ref{t:main_qi_embed}.
The examples are inspired by  \cite[p.253]{MR1886626} and consist of pairs of planar domains, and we will use the obvious identification between $\mathbb{C}$ and $\mathbb{R}^2$ in our notation.

\begin{ex}[Quasi-geodesic]
Set $X=\{(x,y)\in \mathbb{C}:\; 0\leq x\text{ and }0\leq y \leq \pi\}$ and $Y=\{(x,y)\in \mathbb{C}:\;0\leq y\text{ and }x^2+y^2\geq 1\} $. 
 The space $X$ endowed with the two-dimensional Lebesgue measure is of locally $2$-bounded geometry, and the same holds true for $Y$. At large scale, $X$ has linear volume growth. The space $Y$ on the other hand, is a proper $2$-regular $2$-Loewner space. The only assumption of Theorem \ref{t:main_qi_embed} not fulfilled in this situation is the presence of a quasi-geodesic in $X$. A quasiconformal, in fact a conformal, map of $X$ onto $Y$ is provided by the exponential function $f(z)=e^z$. The example shows the necessity of this condition and it also illustrates that one cannot replace the assumption ``$X$ contains a quasi-geodesic'' by ``$Y$ contains a quasi-geodesic''.
\end{ex}

\begin{ex}[Properness]
Set $X=\{(x,y)\in \mathbb{C}:\;  0\leq y \leq \pi\}$ and $Y=\{(x,y)\in \mathbb{C}:\;y\geq 0\}\setminus\{(0,0)\} $ with $f(z)=e^z$, so that $X$ and $Y$ are quasiconformally equivalent. In this example, $X$ does contain a quasi-geodesic and it is further a space of locally $2$-bounded geometry with linear volume growth at large scale. The space $Y$ is a $2$-Loewner space of locally $2$-bounded geometry but not proper as it does not contain the origin.
\end{ex}

\begin{ex}[Asymptotic behavior for the Loewner function]
\label{ex:loewner}
 Set $X:=\{(x,y)\in \mathbb{C}:\;  -1\leq y \leq 1\}$ and $Y=f_{\lambda}(X)$, where $f_{\lambda}$ for a given $\lambda\in (1,2)$ denotes the radial stretch map $f_{\lambda}(z)=z|z|^{\frac{\lambda-1}{2-\lambda}}$,   hence the two spaces are quasiconformally equivalent. Both $X$ and $Y$ are of locally $2$-bounded geometry. The assumption which is not fulfilled in this case is the condition on the asymptotic behavior of $\phi_2$. To see this, notice first that, for some $a>0$,
\begin{equation}\label{eq:shapeY}
Y\subset  \left[-a,a\right]^2 \cup \{(x,y)\in \mathbb{C}:\;- a |x|^{\lambda-1}\leq y\leq a |x|^{\lambda-1}\}.
\end{equation}
The space $Y$ contains the real axis as a quasi-geodesic and \eqref{eq:shapeY} shows that there exist constants $R_0>0$ and $C_0>0$ such that
\begin{displaymath}
\mathcal{L}^2(B(0,r)) \leq C_0 r^{\lambda}\quad \text{for all }r\geq R_0.
\end{displaymath}
Hence it follows from Proposition \ref{p:bdd_loewner} that the $2$-Loewner function of the space $Y$ is bounded.
\end{ex}

\section{Application to the sub-Riemannian Heisenberg and roto-translation groups}\label{Application}

In this section, we are going to prove Corollary \ref{RTvsH} as an application  of Corollary \ref{c:main_thm}, i.e., we show that the sub-Riemannian Heisenberg group and the sub-Riemannian roto-translation group are not quasiconformally equivalent.

\subsection{Sub-Riemannian Lie groups}
We briefly recall some notions from sub-Riemannian geometry in the particular case of Lie groups.
We consider a Lie group $G$ together with a left-invariant Riemannian metric $\langle\cdot,\cdot\rangle$ and with a left-invariant bracket-generating subbundle $\Delta$ of the tangent bundle of $G$.
We equip the groups $G$ with
the Carnot-Carath\'eodory distance $d$ associated to  $\Delta$  and $\langle\cdot,\cdot\rangle$, defined as,
for all $p,q\in G$,
 $$d( p,q):= \inf\left\{ \int_0^1\norm{\dot \gamma(t)} {\rm d}t : \gamma\in C^1([0,1],G), \dot\gamma(t)\in\Delta_{\gamma(t)},  \; ^{\gamma(0)=p}_{  \gamma(1)=q}\right\}.$$
We remark that the above definition only depends on the values of
 $\langle\cdot,\cdot\rangle$
on $\Delta$. Moreover, since $\Delta$ is bracket generating, the distance $d$ is finite, geodesic, and induces the manifold topology.

\subsection{The sub-Riemannian Heisenberg group}

We will choose the following coordinates for the Heisenberg group.
\begin{defn}
The \emph{Heisenberg group} $\mathbb{H}^1$ is $\mathbb{R}^3$ endowed with the group law
\begin{displaymath}
(x,y,t)(x',y',t')=(x+x',y+y',t+t'-2yx'+2xy').
\end{displaymath}
\end{defn}
The left-invariant subbundle $\Delta$ of the tangent-bundle is spanned by the    frame
\begin{displaymath}
X=\frac{\partial}{\partial x}+2y\frac{\partial}{\partial t},\quad  Y=\frac{\partial}{\partial y}-2x\frac{\partial}{\partial t}.
\end{displaymath}
 Notice that $\Delta$ is bracket-generating and is the kernel of the contact form
\begin{displaymath}
\beta= dt - 2ydx + 2xdy
\end{displaymath}

Consider $\langle \cdot,\cdot \rangle$  any inner product that makes $X,Y$ orthonormal, and  define the  Carnot-Carath\'eodory distance $d_{H}$.


In these coordinates, a Haar measure for $\mathbb{H}^1$ is simply the Lebesgue measure. 

\begin{thm}\label{p_loc_geo_Heis}
$(\mathbb{H}^1,\mathcal{L}^3,d_{H})$ is a
proper Ahlfors $4$-regular $4$-Loewner space, hence it is of locally  $4$-bounded geometry.
\end{thm}

\begin{proof}
See Theorem 9.27 together with Theorem 9.10 in \cite{bHe01}.
\end{proof}

\subsection{The sub-Riemannian roto-translation group}

Another example of a sub-Riemannian structure  on $\mathbb{R}^3$ is provided by the roto-translation group, which plays a prominent role in modeling visual perception.

\begin{defn}
The \emph{roto-translation group} $\mathcal{RT}$  is $\mathbb{R}^3$ endowed with the group law
\begin{equation}\label{eq:rt_eq}
(x,y,\theta) \cdot (x',y',\theta'):=  \left(\begin{pmatrix}x\\y\end{pmatrix}+\begin{pmatrix}\cos \theta&-\sin \theta\\ \sin \theta & \cos \theta\end{pmatrix}\begin{pmatrix}x'\\y'\end{pmatrix},\theta+\theta'\right).
\end{equation}
\end{defn}

\begin{rem} Notice that the space $\mathcal{RT}$ defined as above is actually the universal covering space of what is occasionally called  \emph{roto-translation group} in the literature, namely the group of orientation-preserving isometries of $\mathbb{R}^2$, also denoted as $SE(2)$ in the literature. The latter is diffeomorphic to $\mathbb{R}^2 \times S^1$.
\end{rem}

Let
$\Delta$ be the     subbundle generated by the left-invariant vector fields
\begin{displaymath}
X= \cos \theta \frac{\partial}{\partial x}+\sin \theta \frac{\partial}{\partial y},\quad Y=\frac{\partial}{\partial \theta}.
\end{displaymath}
Equivalently,   $\Delta$ is the kernel of the contact form
\begin{displaymath}
\alpha= \sin \theta dx -\cos \theta dy.
\end{displaymath}
Note that $\Delta$ is   bracket-generating since
$$[X,Y]= \sin \theta \frac{\partial}{\partial x}-\cos \theta \frac{\partial}{\partial y},$$
which is linearly independent from $X$ and $Y$ at every point.

Given   an inner product $\langle \cdot,\cdot \rangle$    that makes $X,Y$ orthonormal, we  define the  Carnot-Carath\'eodory distance $d_{RT}$.


For a fixed element  $(x,y,\theta)\in \mathcal{RT}$, let
$L_{(x,y,\theta)}:\mathcal{RT}\to\mathcal{RT}$  be the left translation with respect to the group law defined in \eqref{eq:rt_eq}. Its differential has the form
\begin{displaymath}
dL_{(x,y,\theta)}=\begin{pmatrix}\cos \theta&-\sin \theta&0\\\sin \theta&\cos \theta&0\\0&0&1\end{pmatrix}.
\end{displaymath}
In particular the Jacobian of each left translation is equal to $1$, which proves that the Lebesgue measure $\mathcal{L}^3$  is a Haar measure on $\mathcal{RT}$, in these coordinates.

\subsection{Small and large scale geometry}

At small scale, $\mathcal{RT}$ behaves like the sub-Riemannian Heisenberg group.
Actually, from the differential viewpoint the   subbundle in  $\mathcal{RT}$  is globally equivalent to the one in $\mathbb H^1$.
Kirsi Peltonen made us aware of the following fact.
\begin{lem}
\label{l:contacto} The manifolds $(\mathcal{RT},\alpha)$ and $(\mathbb{H}^1,\beta)$ are globally contactomorphic.
\end{lem}

\begin{proof}
A contactomorphism $f:(\mathcal{RT},\alpha)\to (\mathbb{H}^1,\beta)$ is given in the above coordinates by
\begin{displaymath}
f(x,y,\theta)=(-x\cos \theta  -y\sin \theta , \theta, 4x \sin \theta -4y\cos \theta-2x\theta \cos \theta-2y\theta \sin \theta).
\end{displaymath}
A direct computation shows that this is an invertible map with $f^{\ast}\beta = 4 \alpha$.
\end{proof}

Since now we have a smooth map for which
$ f_*\Delta_{RT}=\Delta_{H}$, we also have that
$ f^*|_{\Delta_{H}} \langle\cdot,\cdot\rangle_{H}|_{\Delta_{H}} $ is a smooth multiple of $ \langle\cdot,\cdot\rangle_{RT}|_{\Delta_{RT}}.$
Hence, with respect to the corresponding sub-Riemannian distances, $f$ is locally Lipschitz.
Thus we have the following two consequences.


\begin{cor}\label{c:biLIp}
There exists a homeomorphism between
$(\mathcal{RT},d_{RT})$ and $(\mathbb{H}^1,d_{H})$ that is   bi-Lipschitz on compact sets.
\end{cor}
In fact, the infinitesimal geometry of our spaces is the same. Namely, $\mathbb{H}^1$ is the metric tangent cone of $\mathcal{RT}$ \cite[p. 52]{MR1421822}.

From Corollary \ref{c:biLIp} we have that there exists a constant $L$ for which the unit ball at the origin in
$\mathcal{RT}$
is $L$-bi-Lipschitz equivalent to some neighborhood $U$ of the origin in
$\mathbb{H}^1$. Since $\mathcal{RT}$ is isometrically homogeneous, {\em any} unit ball  in
$\mathcal{RT}$
is $L$-bi-Lipschitz equivalent to  $U$. Therefore, from Theorem \ref{p_loc_geo_Heis} we deduce the next consequence.

\begin{cor}\label{bdd:geo:RT}
$(\mathcal{RT},d_{RT},\mathcal{L}^3)$ is of locally $4$-bounded geometry.
\end{cor}

At large scale, $\mathcal{RT}$ behaves like the Euclidean space $(\mathbb{R}^3,d_E)$. Since there is a proper and co-compact action of $\mathbb{Z}^3$ by isometries on both $\mathcal{RT}$ and on $\mathbb{R}^3$, and these spaces are proper and geodesic, it follows by the Schwartz-Milnor Lemma that they are quasi-isometric (and quasi-isometric to $\mathbb{Z}^3$ with respect to a word metric). Below we give a more explicit proof.

\begin{prop}\label{id:QI}
The map $\mathrm{id}:(\mathbb{R}^3,d_E)\to (\mathbb{R}^3,d_{RT})$ is a quasi-isometry.
\end{prop}

\begin{proof}
Let us denote
\begin{displaymath}
\Omega:= [0,1)\times [0,1)\times [0,2\pi).
\end{displaymath}
Left translation by an element $g\in \mathbb{Z}^2 \times 2\pi \mathbb{Z}$ with respect to the group law of the roto-translation group coincides with the usual translation by $g$ on the underlying Euclidean space and we have
\begin{displaymath}
\mathbb{R}^3 = \bigcup_{g\in \mathbb{Z}^2 \times 2\pi \mathbb{Z}} g \Omega.
\end{displaymath}
Given two points $p,p' \in \mathbb{R}^3$, let $\gamma:[0,d_{RT}(p,p')]\to (\mathbb{R}^3,d_{RT})$ be an arc-length parametrized geodesic joining them. Moreover, let $n\in \mathbb{N}$ be so that
\begin{displaymath}
d_{RT}(p,p')\leq n < d_{RT}(p,p')+1.
\end{displaymath}
Set $t_i:= \frac{d_{RT}(p,p')}{n} i$, for $i=0,\ldots,n$, and $p_i:=\gamma(t_i)$. In particular, $p_0=p$ and $p_n=p'$. Then
\begin{displaymath}
d_{RT}(p_i,p_{i-1})=d_{RT}(\gamma(t_i),\gamma(t_{i-1}))\leq |t_i-t_{i-1}|=\frac{d_{RT}(p,p')}{n}\leq 1,
\end{displaymath}
since $\gamma$ is a $1$-Lipschitz parametrization. Define
\begin{displaymath}
\lfloor \cdot \rfloor:\mathbb{R}^3 \to \mathbb{Z}^2 \times 2 \pi \mathbb{Z},\quad \lfloor (x,y,z) \rfloor := (\lfloor x \rfloor, \lfloor y \rfloor , 2\pi\lfloor \tfrac{z}{2\pi} \rfloor).
\end{displaymath}
Since the metrics $d_E$ and $d_{RT}$ are left-invariant with respect to the corresponding group laws, we have
\begin{displaymath}
d_{E} (p, \lfloor p \rfloor ) \leq \mathrm{diam}_{E} (\Omega) =: R_{E}\quad \text{and}\quad d_{RT} (p, \lfloor p \rfloor ) \leq \mathrm{diam}_{RT} (\Omega) =: R_{RT},
\end{displaymath}
for all $p\in \mathbb R^3$.
Notice that
\begin{displaymath}
d_{RT}(\lfloor p_i\rfloor, \lfloor p_{i-1} \rfloor)\leq
d_{RT}(\lfloor p_i\rfloor, p_i)+ d_{RT}(p_i,p_{i-1})+
d_{RT}(p_{i-1},\lfloor p_{i-1} \rfloor)\leq 2R_{RT}+1.\end{displaymath}
Hence, we have that
\begin{align*}
d_E(\lfloor p\rfloor, \lfloor p' \rfloor)&\leq \sum_{i=1}^n d_E (\lfloor p_i\rfloor, \lfloor p_{i-1} \rfloor)\leq n \cdot \sup\{d_E(\lfloor p_i\rfloor, \lfloor p_{i-1} \rfloor)\}=: n \cdot M_E,
\end{align*}
where the supremum is taken over all points with $d_{RT}(\lfloor p_i\rfloor, \lfloor p_{i-1} \rfloor)\leq 2R_{RT}+1$. As there are only finitely many points of $\mathbb{Z}^2 \times 2\pi \mathbb{Z}$ inside a ball centred at the origin, the finiteness of $M_E$ follows by left translation. By the choice of $n$, we conclude
\begin{displaymath}
d_E(\lfloor p\rfloor, \lfloor p' \rfloor) \leq M_E d_{RT}(p,p')+ M_E
\end{displaymath}
and thus
\begin{displaymath}
d_{E}(p,p') \leq d_E(\lfloor p\rfloor, \lfloor p' \rfloor) + 2 R_{E}\leq M_E d_{RT}(p,p')+ (M_E+ 2 R_E).
\end{displaymath}
Hence, we got one of the quasi-isometry bounds. The other one is obtained  with the same line of arguments 
 when the roles of $E$ and $RT$ are reverted.
\end{proof}

\begin{cor}\label{cubic:vol}
In the sub-Rie{\-}mannian roto-translation group $\mathcal{RT}$
the volume in the large is cubic, i.e., condition \eqref{lower:volume:growth} holds for $N=3$.
\end{cor}

\begin{proof}
Since $(\mathbb{R}^3,d_{RT})$ and the Euclidean space $\mathbb{R}^3$ are quasi-isometric, there exist constants $L\geq 1$ and $b\geq 0$ such that
\begin{displaymath}
B_{\mathcal{RT}}(p,r) \subseteq B_E(p,Lr+b),\quad\text{for all }p\in \mathbb{R}^3,\; r>0.
\end{displaymath}
By monotonicity the claim follows  from the cubic volume growth of Euclidean balls in $\mathbb{R}^3$ with respect to the Lebesgue measure.
\end{proof}

\begin{proof}[Proof of Corollary \ref{RTvsH}]
We  apply
 Corollary \ref{c:main_thm} with $X=\mathcal{RT}
$ and $Y=\mathbb{H}^1$.
Both $X$ and $Y$ have locally $4$-bounded geometry by Theorem \ref{p_loc_geo_Heis} and Corollary \ref{bdd:geo:RT}, respectively.
The space $X$ is geodesic and has, by Corollary \ref{cubic:vol}, property \eqref{lower:volume:growth} with exponent $3$, which is less than $4$.
Since $X$ is quasi-isometric to the Euclidean space by Proposition \ref{id:QI}, any line in $\mathbb R^3$ is a  quasi-geodesic for $d_{RT}$.
By Theorem~\ref{p_loc_geo_Heis}, the space $Y$ is a
proper Ahlfors $4$-regular $4$-Loewner space.
Thus Corollary \ref{c:main_thm} gives our claim. 
\end{proof}

%
%
%
%


\bibliographystyle{alpha}
\bibliography{qs}
\end{document}